\documentclass[12pt]{article}

\usepackage{amsmath,amssymb}
\usepackage{amsthm}
\usepackage{mathrsfs}
\usepackage[margin=1in]{geometry}
\usepackage{ifpdf}
\ifpdf
  \usepackage[hidelinks]{hyperref}
\else
  \usepackage[dvipdfmx,hidelinks]{hyperref}
\fi

\numberwithin{equation}{section}
\allowdisplaybreaks[3]

\title{On Weakly Separable Polynomials and Weakly Quasi-separable Polynomials over Rings}
\author{Satoshi Yamanaka\\[0.25em]
\small Department of Integrated Science and Technology\\
\small National Institute of Technology, Tsuyama College\\
\small Okayama 708-8509, JAPAN\\
\small \texttt{yamanaka@tsuyama.kosen-ac.jp}}
\date{}

\newtheorem{thm}{\quad Theorem}[section]
\newtheorem{lem}[thm]{\quad Lemma}

\newtheorem{prop}[thm]{\quad Proposition}
\newtheorem{cor}[thm]{\quad Corollary}
\newtheorem{dfn}[thm]{\quad Definition}
\theoremstyle{remark}
\newtheorem*{exm}{Example}
\newtheorem*{rmk}{Remark}

\begin{document}

\maketitle

%\begin{center}
%\emph{Dedicated to Professor Takasi Nagahara on his 85th birthday}
%\end{center}

\begin{abstract}
Separable extensions of  noncommutative rings have  already been studied extensively.  
Recently, N. Hamaguchi and A. Nakajima introduced the notions of weakly separable  extensions and  weakly quasi-separable extensions. 
They studied weakly separable polynomials and weakly quasi-separable polynomials in the 
case that the coefficient ring is  commutative.     
The purpose of this paper is to give some improvements and  generalizations of  
Hamaguchi and Nakajima's results. 
We shall characterize a weakly separable polynomial $f(X)$ over a commutative ring by using  its derivative $f'(X)$ 
and its discriminant $\delta(f(X))$. 
Further, 
we shall 
try to give necessary and sufficient conditions for weakly separable polynomials in skew polynomial rings 
in the case that the coefficient ring  is  noncommutative.
\end{abstract}

\vspace{0.5cm}
\noindent
{\small {\bf Note for the arXiv version.} This manuscript corresponds to the article ``On Weakly Separable Polynomials and Weakly Quasi-separable Polynomials over Rings,'' published in {\it Math. J. Okayama Univ.} {\bf 58} (2016), 169--182. The version of record is available at \href{https://doi.org/10.18926/mjou/53924}{10.18926/mjou/53924}.}
\vspace{0.5cm}



\section{Introduction}

Throughout this paper, $A/B$ will represent a ring extension with common identity 1. 
Let $M$ be  an $A$-$A$-bimodule, and $x$, $y$ arbitrary elements in $A$. 
Then an additive map $\delta$ is called a {\it $B$-derivation} of $A$ to $M$ if 
$\delta(xy) = \delta(x)y + x\delta(y)$  and $\delta(\alpha)=0$ for any $\alpha \in B$. 
Moreover, $\delta$ is called {\it central} if $\delta(x)y=y\delta(x)$, and 
$\delta$ is called {\it inner} if $\delta(x)=mx -xm$ for some fixed element $m \in M$. 
We say that a ring extension $A/B$ is {\it separable}
if the $A$-$A$-homomorphism of $A \otimes_BA$ onto $A$
defined by $a \otimes b \mapsto ab$ splits. 
It is well known that $A/B$ is separable if and only if for any $A$-$A$-bimodule $M$, every $B$-derivation of $A$ to $M$ is inner 
(cf. \cite[Satz 4.2]{E}). 
In \cite{N}, Y. Nakai introduced the notion of a quasi-separable extension of commutative rings  by using the module differentials, 
and in the noncommutative case, it was characterized by H. Komatsu \cite[Lemma 2.1]{Ko} as follows : 
$A/B$ is  {\it quasi-separable} if and only if for any $A$-$A$-bimodule $M$, every central $B$-derivation of $A$ to $M$ is zero. 
Recently in \cite{HN}, N. Hamaguchi and A. Nakajima gave the following definitions as  generalizations  
of  separable extensions and quasi-separable extensions. 
\begin{dfn}{\rm \cite[Definition 2.1]{HN}}
{\rm (1) $A/B$ is called}  {\it weakly separable} {\rm if every $B$-derivation of $A$ to $A$ is inner.}\par
{\rm (2) $A/B$  is called}  {\it weakly quasi-separable}   {\rm if every central $B$-derivation  of $A$ to $A$ is zero.}
\end{dfn}
Obviously, a separable extension is weakly separable and a quasi-separable extension is weakly quasi-separable. 
Moreover, a separable extension is quasi-separable by \cite[Theorem 2.4]{Ko}.

Let $B$ be a ring, $\rho$ an automorphism of $B$, $D$ a $\rho$-derivation, that is, $D$ is an additive endomorphism of $B$
 such that $D(\alpha \beta) = D(\alpha)\rho(\beta) +\alpha D(\beta)$ for any $\alpha, \beta \in B$.    
$B[X;\rho, D]$ will mean  the skew polynomial ring in which
the multiplication is given by  $\alpha X = X\rho(\alpha) + D(\alpha)$ for any  $\alpha \in B$. 
We write $B[X; \rho] = B[X;\rho, 0]$   and $B[X; D] = B[X; 1, D]$. 
By $B[X;\rho, D]_{(0)}$ we denote the set of all monic polynomials $g$ 
in $B[X;\rho, D]$ such that $gB[X;\rho, D]=B[X;\rho, D]g$. 
Let $f$ be in $B[X;\rho, D]_{(0)}$. Then the residue ring $B[X;\rho, D]/fB[X;\rho,D]$ is a free ring extension of $B$.   
We say that $f$ is a {\it separable} (resp. {\it weakly separable}, {\it weakly quasi-separable}) {\it polynomial} in $B[X;\rho, D]$ 
if $B[X;\rho, D]/fB[X;\rho,D]$ is separable  (resp. weakly separable, weakly quasi-separable) over $B$.

Separable polynomials in skew polynomial rings have been extensively studied by 
K. Kishimoto, 
T. Nagahara, Y. Miyashita, and S. Ikehata (see References). 
T. Nagaraha studied separable polynomials over a commutative ring 
and  separable polynomials of degree 2 in skew polynomial rings thoroughly.  
He characterized a separable polynomial $f(X)$ over a commutative ring by using its derivative $f'(X)$ and 
its discriminant $\delta(f(X))$ as follows : 
\begin{prop}\label{N}\cite[Theorem 2.3]{N1}
Let $B$ be a commutative ring, and $f(X)$ a monic polynomial in $B[X]$. 
Then the following are equivalent.\par
$(1)$ $f(X)$ is separable in $B[X]$.\par 
$(2)$ $f'(X)$ is invertible in $B[X]$ modulo $(f(X))$.\par
$(3)$ $\delta(f(X))$ is  invertible in $B$.
\end{prop}
Concerning separable polynomials in skew polynomial rings, Y. Miyashita proved the following. 
\begin{prop}\label{M}{\rm \cite[Theorem 1.8]{M}}
Let $f=X^m+X^{m-1}a_{m-1}+\cdots + Xa_1 + a_0$ be in $B[X;\rho,D]_{(0)}$.
We set $A=B[X;\rho,D]/fB[X;\rho,D]$ and $x=X+fB[X;\rho,D]$. 
Then $f$ is separable in $B[X;\rho,D]$ if and only if 
there exists $h \in A$ such that $\rho^{m-1}(\alpha)h=h\alpha$ for any $\alpha \in B$ and 
$\sum_{j=0}^{m-1}y_j h x^j=1$, where $y_j=x^{m-j-1}+x^{m-j-2}a_{m-1}+ \cdots +xa_{j+2}+a_{j+1}$ $(0 \leq j \leq m-2)$ 
and $y_{m-1}=1$.  
\end{prop}
Recently in \cite{YI1}, the author and S. Ikehata gave alternative  proofs  of  the above proposition  
in $B[X;\rho]$ and $B[X;D]$, respectively. 
In addition, 
S. Ikehata gave some refinements and sharpenings of Miyashita's results concerning separable polynomials in skew polynomial rings 
(cf. \cite{I1}, \cite{I2}, \cite{I00}, \cite{I000}).  
  
In \cite{HN}, N. Hamaguchi and A. Nakajima studied weakly separable polynomials over a commutative ring. 
They also studied weakly separable polynomials and weakly quasi-separable polynomials in  skew polynomial rings 
$B[X;\rho]$ and $B[X;D]$ when $B$ is an integral domain. 
The purpose of this paper is to give some refinements and  generalizations of their results which were obtained in \cite{HN}.

In section 2,  we treat weakly separable polynomials over a commutative ring. 
As mentioned above,  a separable polynomial $f(X)$ in $B[X]$ has a close relationship with 
the invertibilities of its derivative $f'(X)$ and   its discriminant $\delta(f(X))$. 
We shall characterize the weakly separability of $f(X)$ in terms of the properties of $f'(X)$ and $\delta(f(X))$.

In section 3, we study weakly separable polynomials and weakly quasi-separable polynomials 
in skew polynomial rings. 
When $B$ is an integral domain, N. Hamaguchi and A. Nakajima gave 
necessary and sufficient conditions for weakly separable polynomials in $B[X;\rho]$ and $B[X;D]$ 
(cf. \cite[Theorem 4.1.4 and Theorem 4.2.3]{HN}).  
We shall try to give sharpenings of their results  for a noncommutative coefficient ring $B$. 
Moreover, we shall study the relationship between the separability and the  weakly separability  
in skew polynomial rings $B[X;\rho]$ and $B[X;D]$, respectively.

\smallskip

\section{Weakly separable polynomials  over a commutative ring}
In this section, we shall study  weakly separable polynomials over a commutative ring. 
It is well known that a (noncommutative) ring extension $A/B$ is separable 
if and only if there exists $\sum_j x_j \otimes y_j \in (A \otimes_B A)^A$ such that 
$\sum_j x_jy_j=1$, where $(A\otimes_BA)^A=\{ \mu  \in A\otimes_BA \, | \, x \mu = \mu x \ {\rm for \ any} \ x \in A\}$
 (cf. \cite[Definition 2]{HS}). 
First we shall prove the following.
\begin{lem}\label{lem1}
Let $A/B$ be a commutative ring extension. 
If  there exists $\sum_j x_j \otimes y_j \in (A \otimes_B A)^A$ such that 
$\sum_j x_jy_j$ is  a non-zero-divisor in $A$, 
then $A/B$ is weakly separable.  
\end{lem}

\begin{proof} 
Let $D$ be a $B$-derivation of $A$, and 
$\sum_j x_j \otimes y_j$ in $(A \otimes_B A)^A$ such that $\sum_j x_jy_j$ is  a non-zero-divisor in $A$. 
Then we consider the following $A$-$B$-homomorphisms :
\begin{center}
\begin{tabular}{ccccc}
$A \otimes_B A$  & $\longrightarrow$ & $A \otimes_BA$ & $\longrightarrow$ & $A$\\

$x\otimes y$ &  $\longmapsto$ & $x\otimes D(y)$ & $\longmapsto$ & $xD(y)$
\end{tabular}
\end{center}
Since $\alpha\sum_jx_j \otimes y_j = \sum_jx_j \otimes y_j \alpha$ for any $\alpha \in A$, we have 
\begin{align*}
\alpha\sum_j x_j D(y_j) = \sum_j x_j D(y_j \alpha) = \sum_j x_jD(y_j)\alpha + \sum_j x_j y_j D(\alpha), 
\end{align*}
and hence $\sum_j x_j y_j D(\alpha)=0$. 
Since $\sum_j x_j y_j$ is a non-zero-divisor in $A$, we obtain $D(\alpha)=0$. 
Therefore $A/B$ is weakly separable. 
This completes the proof. 
\end{proof}

\smallskip

\begin{exm}{\rm
Let $B$ be a ring, $G$ a finite group of order $n$, and $A=B[G]$, that is,  $A$ is a group ring of $G$ over $B$.    
Then it is easily seen that $\sum_{g \in G} g \otimes g^{-1} \in (A \otimes_B A)^A$, 
and hence   $A/B$ is separable  if $n \, (=\sum_{g \in G} gg^{-1})$ is an invertible element. 
When $B$ is commutative and $G$ is abelian, if $n$ is a non-zero-divisor in $A$  
then  $A/B$ is weakly separable   by Lemma \ref{lem1}. 
Moreover, it is also true when $B$ is  noncommutative  and $G$ is abelian.  
In fact, for any $B$-derivation $D$ of $A$ and   $\alpha \in A$, we see that 
$$
\alpha \sum_{g \in G} g D(g^{-1}) = \sum_{g \in G} g D(g^{-1}\alpha)= \sum_{g \in G} g D(g^{-1}) \alpha + \sum_{g \in G} g g^{-1} D(\alpha).
$$
Noting that $D(g^{-1})$ is in the center of $A$, we have $nD(\alpha)=0$. 
Hence if $n$ is a non-zero-divisor in $A$ then $D=0$ . 
}\end{exm}

\smallskip

Now, let $B$ be a commutative ring. %, and $f(X)$ a polynomial in $B[X]$. 
For a monic polynomial $f(X) \in B[X]$, $f'(X)$ and $\delta(f(X))$ will mean the derivative of  $f(X)$
and the discriminant of $f(X)$, respectively.   %and $f(X)=X^m + \sum_{j=0}^{m-1}X^j a_j \in B[X]$. 
Then a separable polynomial $f(X)$ in $B[X]$ 
is characterized  as Proposition \ref{N}  by using $f'(X)$ and $\delta(f(X))$.    
In \cite{HN}, N. Hamaguchi and A. Nakajima studied the weakly separability of a polynomial $f(X)=X^m-Xa-b$ in $B[X]$.  
They proved that 
$f(X)=X^m-Xa-b$ is weakly separable in $B[X]$ if and only if 
$\delta(f(X))$ is  a non-zero-divisor in $B$, or equivalently, 
$f'(X)$ is  a  non-zero-divisor in $B[X]$ modulo $(f(X))$ (cf \cite[Theorem 3.1 and Corollary 3.2]{HN}). 
Now we shall give a sharpening of their result. %for a general polynomial $f(X)$.    
\begin{thm}
Let $B$ be a commutative ring, and $f(X)$ a monic polynomial in $B[X]$. 
 The following are equivalent.\par
$(1)$ $f(X)$ is weakly separable in $B[X]$.\par 
$(2)$ $f'(X)$ is a non-zero-divisor in $B[X]$ modulo $(f(X))$.\par
$(3)$ $\delta(f(X))$ is  a non-zero-divisor in $B$.
\end{thm}

\begin{proof} It is already known  that (2) and (3) are equivalent by \cite[Theorem 1.3]{N1}. 
We shall show that (1) and (2) are equivalent. 
Let $f(X)=X^m+X^{m-1}a_{m-1}+\cdots+Xa_1+a_0$ be in $B[X]$, $A=B[X]/(f(X))$, $x=X+(f(X))$, and $f'(x)=f'(X)+(f(X)) \in A$.\par
$(2) \Longrightarrow (1).$ It follows from \cite[Lemma 2.1]{YI1} (or \cite[Lemma 3.1]{YI1}) that 
$$
(A \otimes_B A)^A=\{ \sum_{j=0}^{m-1} y_j h \otimes x^j \, | \, h \in A\},
$$ 
where 
$y_j = x^{m-j-1} + x^{m-j-2}a_{m-1} + \cdots + xa_{j+2} + a_{j+1} \  (0 \leq j \leq m-2)$ 
and $y_{m-1}=1$. 
In particular, we see that $\sum_{j=0}^{m-1} y_j \otimes x^j \in (A \otimes_B A)^A$. 
Noting that $f'(x)=\sum_{j=0}^{m-1}y_jx^j$ and $f'(x)$ is a non-zero-divisor in $A$, 
$f(X)$ is weakly separable in $B[X]$ by Lemma \ref{lem1}.

$(1) \Longrightarrow (2).$ Assume that $f'(x)g(x)=0$ for some $g(x) \in A$. 
We can then construct a $B$-derivation $D$ of $A$ such that  $D(x)=g(x)$
because 
 $D(f(x))=f'(x)D(x)=f'(x)g(x)=0$. 
Since $f(X)$ is weakly separable, we have $g(x)=0$.  
This completes the proof. 
\end{proof}

\smallskip

\section{Weakly separable polynomials in skew polynomial rings}
In \cite{HN}, N. Hamaguchi and A. Nakajima characterized weakly separable polynomials and 
weakly quasi-separable polynomials in skew polynomial rings $B[X;\rho]$ and $B[X;D]$ 
when $B$ is an integral domain. 
In this section, we shall generalize their results 
for a noncommutative coefficient ring $B$. 

We shall use the following conventions :\par
$Z=$ the center of $B$\par
$V_A(B)=$ the centralizer of  $B$ in $A$\par 
$B^\rho=\{ \alpha \in B \, | \, \rho(\alpha)=\alpha \}$\par% and $Z^\rho=\{ z \in Z \, | \, \rho(z)=z \}$\par
$B^D=\{ \alpha \in B \, | \, D(\alpha)=0 \}$ and $Z^D=Z \cap B^D$\par
$D(B)=\{ D(\alpha) \, | \, \alpha \in B \}$

\subsection{Automorphism type} 
In this section, we consider a polynomial $f$  in  $B[X;\rho]_{(0)}$ %$B[X;\rho]_{(0)} \cap B^{\rho}[X]$
  of the form 
\begin{align*}
f=X^m+X^{m-1}a_{m-1} + \cdots + Xa_1+a_0=\sum_{j=0}^{m} X^j a_j \  \ (a_m=1, \ m \geq 2).
\end{align*}
We set $A=B[X;\rho]/fB[X;\rho]$, and $x=X+fB[X;\rho] \in A$.  
By \cite[Lemma 1.3]{I1}, $f$ is in $B[X;\rho]_{(0)}$ if and only if 
\begin{align*}
\left\{\begin{array}{l}
\alpha a_j = a_j \rho^{m-j}(\alpha) \ \ (\alpha \in B, \ 0 \leq j \leq m-1),\\
\rho(a_j)-a_j = a_{j+1} (\rho(a_{m-1})-a_{m-1}) \ (0 \leq j \leq m-2),\\
a_0(\rho(a_{m-1})-a_{m-1})=0.
\end{array}\right.
\end{align*}
Now, we let $f$ be in $B[X;\rho]_{(0)} \cap B^{\rho}[X]$. 
Then there is  an automorphism $\tilde{\rho}$ of $A$ which is naturally induced by $\rho$, 
that is, $\tilde{\rho}$ is defined by 
$\tilde{\rho}(\sum_{j=0}^{m-1}x^j c_j)=\sum_{j=0}^{m-1}x^j\rho(c_j)$.   
We write $J_{\rho^k}=\{ h \in A \, | \, \alpha h = h \rho^k(\alpha) \ (\alpha \in B)\}$ $(k \geq 1)$, 
$V=V_A(B)$, and $V^{\tilde{\rho}}=\{ h \in V \, | \, \tilde{\rho}(h)=h\}$. 
Then we consider a $V^{\tilde{\rho}}$-$V^{\tilde{\rho}}$-homomorphism  $\tau : J_\rho \longrightarrow J_{\rho^m}$ 
defined by  
\begin{align*}
\tau(h) &=x^{m-1}\sum_{j=0}^{m-1}\tilde{\rho}^{j}(h) + x^{m-2} \sum_{j=0}^{m-2}\tilde{\rho}^j(h)a_{m-1}
 + \cdots x\big\{\tilde{\rho}(h)+h\big\}a_2 + h a_{1}\\
&= \sum_{k=0}^{m-1}x^k \sum_{j=0}^{k}\tilde{\rho}^{j}(h)a_{k+1}. %\  \ (a_m=1).
\end{align*}

First we shall prove the following.  

\begin{lem}\label{lemr1}
If $\delta$ is a $B$-derivation of $A$, then $\delta(x) \in J_\rho$ and $\tau(\delta(x)) =
0$. 
Conversely, if $g \in J_\rho$ with $\tau(g) = 0$, then there exists a $B$-derivation
$\delta$ of $A$ such that $\delta(x) = g$.
\end{lem}

\begin{proof} 
Let $\delta$ be a $B$-derivation of $A$. 
It can be easily seen that $\alpha \delta(x)=\delta(x)\rho(\alpha)$ for any $\alpha \in B$. 
Since $\delta(x^k)=x^{k-1}\sum_{i=0}^{k-1}\tilde{\rho}^i(\delta(x))$  $(k \geq 2)$, 
we have 
$$
0=\delta(\sum_{k=0}^{m}x^k a_k)=\sum_{k=0}^{m-1} \delta(x^{k+1}) a_{k+1} = 
\sum_{k=0}^{m-1}  x^k\sum_{j=0}^{k}\tilde{\rho}^j(\delta(x)) a_{k+1} 
= \tau(\delta(x)).
$$

Conversely, Let $g=g_0 +fB[X;\rho]$ $(g_0 \in B[X;\rho])$ be in $J_\rho$ such that $\tau(g)=0$.  
Since $\alpha g_0 = g_0 \rho(\alpha)$ $(\alpha \in B)$, 
we can define a $B$-derivation $\delta^*$ of $B[X;\rho]$ such that  $\delta^*(X)=g_0$. 
Moreover, since $\tau(g)=0$, it is easy to see that $\delta^*(f) \in fB[X;\rho]$. 
Hence there is a $B$-derivation $\delta$ of $A$ such that $\delta(x)=g$  
which is naturally  induced by $\delta^*$. 
This complete the proof. 
\end{proof}

\smallskip

Now we shall give a sharpening of \cite[Theorem 4.1.4]{HN} 
\begin{thm}\label{thmr1}
Let $f=X^m+X^{m-1}a_{m-1} + \cdots + Xa_1+a_0$ be in $B[X;\rho]_{(0)} \cap B^{\rho}[X]$. 
Then $f$ is weakly separable in $B[X;\rho]$ if and only if 
$$
\{ g \in J_\rho \, | \, \tau(g)= 0\} =\{ x(\tilde{\rho}(h)-h) \, | \, h \in V  \}.
$$
\end{thm}

\begin{proof} Assume that 
$\{ g \in J_\rho \, | \, \tau(g)= 0\} =\{ x(\tilde{\rho}(h)-h) \, | \, h \in V  \}$, 
and let $\delta$ be a $B$-derivation of $A$. 
Then it follows from Lemma \ref{lemr1} that  
$\delta(x) \in\{ g \in J_\rho \, | \, \tau(g)= 0\}$, and hence 
$\delta(x)=x(\tilde{\rho}(h)-h)=hx-xh$ for some $h \in V$. 
Then it is easy to see that $\delta(w) = hw-wh$ for any $w \in A$. 
Therefore $\delta$ is inner. 

Conversely, assume that $f$ is weakly separable, and 
let $g$ be an element in $J_\rho$ such that  $\tau(g)=0$. 
Then we can define a $B$-derivation $\delta$ of $A$ such that $\delta(x)=g$ by Lemma \ref{lemr1}.  
Since $f$ is weakly separable, we obtain 
$g=\delta(x)=hx-xh=x(\tilde{\rho}(h)-h)$ for some $h \in V$. 
Thus $\{ g \in J_\rho \, | \, \tau(g)= 0\} \subset \{ x(\tilde{\rho}(h)-h) \, | \, h \in V  \}$. 
It is easy to see that   $\{ g \in J_\rho \, | \, \tau(g)= 0\} \supset \{ x(\tilde{\rho}(h)-h) \, | \, h \in V  \}$. 
In fact, generally $hx^k=x^k\tilde{\rho}^{k}(h)$ $(k \geq 1)$ for any $h \in A$. 
This implies  
$\sum_{k=0}^{m-1}x^k ( \tilde{\rho}^k(h) -\tilde{\rho}^m (h))a_k =0 $ for any $h \in V$, and hence we have 
\begin{align*}
\tau(x(\tilde{\rho}(h)-h)) &= \sum_{k=0}^{m-1}x^k \sum_{j=0}^k \tilde{\rho}^j (x(\tilde{\rho}(h)-h)) a_{k+1}\\
&= x^m \sum_{j=0}^{m-1}\tilde{\rho}^j (\tilde{\rho}(h)-h) + \sum_{k=0}^{m-2}x^{k+1}\sum_{j=0}^{k}\tilde{\rho}^j (\tilde{\rho}(h)-h)a_{k+1}\\
&= (-\sum_{k=0}^{m-1}x^ka_k)(\tilde{\rho}^m(h)-h) + \sum_{k=1}^{m-1}x^{k}(\tilde{\rho}^k(h)-h)a_{k}\\
&= \sum_{k=0}^{m-1} x^k (\tilde{\rho}^k(h) - \tilde{\rho}^m(h) )a_{k}\\
&= 0. 
\end{align*}
This completes the proof. 
\end{proof}

\smallskip

As a direct consequence of Theorem \ref{thmr1}, we have the following. 

\begin{cor}\cite[Theorem 4.1.4 (ii)]{HN}
Let $B$ be an integral domain, $m$ the order of $\rho$, and $f=X^m-u$ $(u \neq 0)$ in $B[X;\rho]_{(0)}$. 
Then $f$ is weakly separable in $B[X;\rho]$ if and only if 
$$
\{ b \in B\ \, | \, \sum_{j=0}^{m-1}\rho^j(b)=0 \} = \{ \rho(c)-c \, | \, c \in B \}. 
$$
\end{cor}

\begin{proof} Since $B$ is an integral domain, then one easily see that 
$J_\rho=\{ xb \, | \, b \in B \}$ and $V=B$. 
If $\tau(xb)=0$ for any $b \in B$, then  we have 
$$
0=\tau(xb)=x^{m-1} \sum_{j=0}^{m-1}\tilde{\rho}^{j}(xb) = u \sum_{j=0}^{m-1} \rho^j(b). 
$$
This means $\sum_{j=0}^{m-1}\rho^j(b)=0$. 
Hence we obtain the assertion by Theorem \ref{thmr1}. 
This completes the proof. 
\end{proof}

\smallskip
 
In virtue of Theorem \ref{thmr1}, we obtain the following theorem 
concerning the relationship between 
the separability and the weakly separability    
in $B[X;\rho]$. 

\begin{thm}\label{thmr2}
Let $m$ be the order of $\rho$, $f=X^m+X^{m-1}a_{m-1} + \cdots + Xa_1+a_0$  in $B[X;\rho]_{(0)} \cap B^{\rho}[X]$, 
 $C(A)$  a center of $A$,   
and  $I_x$  an inner derivation of $A$ by $x$ $($that is, $I_x(h) = hx-xh$ for any $h \in A)$.\par
$(1)$ $f$ is weakly separable in $B[X;\rho]$ if and only if 
the following sequence of $V^{\tilde{\rho}}$-$V^{\tilde{\rho}}$-homomorphisms is exact: 
$$
0 \longrightarrow
% {\rm Ker} \, x(\tilde{\rho}-1)
C(A)  \overset{{\rm inj}}{\longrightarrow}  V 
\overset{I_x}{\longrightarrow} J_\rho \overset{\tau}{\longrightarrow} V^{\tilde{\rho}}.
$$ 

$(2)$ $f$ is separable in $B[X;\rho]$ if and only if the following sequence 
of $V^{\tilde{\rho}}$-$V^{\tilde{\rho}}$-homomorphisms is exact: 
$$
0 \longrightarrow
% {\rm Ker} \, x(\tilde{\rho}-1)
C(A)  \overset{{\rm inj}}{\longrightarrow}  V 
\overset{I_x}{\longrightarrow} J_\rho \overset{\tau}{\longrightarrow} V^{\tilde{\rho}}
\longrightarrow 0.
$$
\end{thm}

\begin{proof} It is easily seen that ${\rm Im} \, \tau \subset V^{\tilde{\rho}}$ 
because $\tilde{\rho}^j(h)a_j=ha_j$ $(0 \leq j \leq m-1)$.\par
(1) It is obvious by Theorem \ref{thmr1}.\par 
(2) If $f$ is separable then $f$ is always weakly separable, and therefore it suffices to show that ${\rm Im} \, \tau =V^{\tilde{\rho}}$. 
By Proposition \ref{M}, $f$ is separable in $B[X;\rho]$ if and only if 
 $h \in A$ such that $\rho^{m-1}(\alpha)h=h\alpha$  for any $\alpha \in B$ and 
$\sum_{j=0}^{m-1}y_j h x^j=1$, where 
$y_j=x^{m-j-1}+x^{m-j-2}a_{m-1}+ \cdots +xa_{j+2}+a_{j+1}$  $(0 \leq j \leq m-2)$  and $y_{m-1}=1$. 
It is obvious that $h \in J_\rho$. 
Noting that $y_jx^j=\sum_{k=j}^{m-1}x^ka_{k+1}$, we obtain 
\begin{align*}
1=\sum_{j=0}^{m-1}y_jx^j \tilde{\rho}^j(h)
=\sum_{j=0}^{m-1} \{ \sum_{k=j}^{m-1}x^ka_{k+1} \}\tilde{\rho}^j(h)
= \sum_{k=0}^{m-1}x^k \sum_{j=0}^{k}\tilde{\rho}^j(h) a_{k+1}
= \tau(h). 
\end{align*}
This implies ${\rm Im} \, \tau =V^{\tilde{\rho}}$ because $\tau$ is a $V^{\tilde{\rho}}$-homomorphism. 
This completes the proof.  
\end{proof}

\smallskip

\begin{rmk}
In this section, we assumed that $f$  is in $B[X;\rho]_{(0)} \cap B^\rho[X]$. 
However, in general case,  a polynomial which is in $B[X;\rho]_{(0)}$ is not always in $B^\rho[X]$. 
Concerning this, we have already known by \cite[Corollary 1.5]{I1}
that if $B$ is a semiprime ring, then every polynomial in $B[X;\rho]_{(0)}$ is in  $C(B^\rho)[X]$, 
where $C(B^\rho)$ is the center  of  $B^{\rho}$. 
\end{rmk}

\smallskip

At the end of this section, we shall mention briefly on  
weakly quasi-separable polynomials in $B[X;\rho]$.  
When $B$ is an integral domain,  N. Hamaguchi and A. Nakajima 
proved that every polynomial in $B[X;\rho]_{(0)}$ 
is weakly quasi-separable (cf. \cite[Theorem 4.1.1]{HN}). 
More precisely, they showed that  it is true for a commutative ring $B$ 
when $\rho \neq 1$ and $\{\rho(c)-c \, | \, c \in B \}$ contains a non-zero-divisor in $B$. 
For an arbitrary ring $B$, we have  the following. 

\begin{prop}
$(1)$ If $\rho\neq1$ and $\{ \rho(c)-c \, | \, c \in B\}$ contains a non-zero divisor, then every polynomial in $B[X;\rho]_{(0)}$ 
is weakly quasi-separable.\par
$(2)$ Let $f=X^m-u$ be in  $B[X;\rho]_{(0)}$. If $m$ and $u$ are non-zero-divisors in $B$, then 
$f$ is weakly quasi-separable in  $B[X;\rho]$.
\end{prop}

\begin{proof} (1) Let $g$ be in $B[X;\rho]_{(0)}$,  $\delta$ a central $B$-derivation of $B[X;\rho]/gB[X;\rho]$, 
and $x=X+gB[X;\rho] \in B[X;\rho]/gB[X;\rho]$. 
Then, for any $\alpha \in B$, we see that $\delta(x)(\rho(\alpha) -\alpha)=0$. 
Hence if  $\{ \rho(c)-c \, | \, c \in B\}$ contains a non-zero divisor, then $\delta=0$.

(2) Let $\delta$ be a central $B$-derivation of $A$ such that $\delta(x) = \sum_{j=0}^{m-1}x^jd_j$. 
Then an easy induction shows that $\delta(x^k)=kx^{k-1}\delta(x)$ for $k \geq 1$.  
Noting that $x^m=u$, we obtain  
$$
0=\delta(x^m)=mx^{m-1}\sum_{j=0}^{m-1}x^j d_j = mx^{m-1} d_0 + m \sum_{j=0}^{m-2}x^j u d_{j+1}.
$$
Thus $md_0=0$ and $mud_j=0$ $(1 \leq j \leq m-1)$, 
and it is obvious that $\delta=0$ if $m$ and $u$ are  non-zero-divisors in $B$. 
This completes the proof. 
\end{proof}

\smallskip

\subsection{Derivation type} 
In this section, let $B$ be of prime characteristic $p$, 
and we consider  a $p$-polynomial $f$ in $B[X;D]_{(0)}$  of the form 
\begin{align*}
f &=X^{p^e}+X^{p^{e-1}}b_e+ \cdots +X^pb_{2} + Xb_1 + b_0 = \sum_{j=0}^{e}X^{p^j}b_{j+1} +b_0 \  (b_{e+1}=1).
\end{align*}
We set $A=B[X;D]/fB[X;D]$, and $x=X+fB[X;D] \in A$. 
By \cite[Corollary 1.7]{I1}, 
$f$ is in  $B[X;D]_{(0)}$ if and only if 
\begin{align*}
\left\{\begin{array}{l}
b_0 \in B^D, \ \ b_{j+1} \in Z^D \  ( 0 \leq j \leq e-1),\\
\sum_{j=0}^{e}D^{p^j}(\alpha)b_{j+1} = b_0 \alpha - \alpha b_0 \ \ (\alpha \in B).
\end{array}\right.
\end{align*}
Since $f$ is in $B^D[X]$, there  is  a derivation $\tilde{D}$ of $A$ which is naturally induced by $D$, 
that is, $\tilde{D}$ is defined by $\tilde{D}(\sum_{j=0}^{p^e-1}x^j c_j)=\sum_{j=0}^{p^e-1}x^jD(c_j)$.  
We write $V=V_A(B)$,  $\tilde{D}(V)=\{ \tilde{D}(h)  \, | \, h \in V\}$, 
and $V^{\tilde{D}}= \{ v \in V \, | \, \tilde{D}(v)=0 \}$. 
Then we consider a $V^{\tilde{D}}$-$V^{\tilde{D}}$-homomorphism $\tau : V \longrightarrow V^{\tilde{D}}$ defined by   
\begin{align*}
\tau(h) &= \tilde{D}^{p^e-1}(h) +\tilde{D}^{p^{e-1}-1}(h)b_{e} + \cdots + \tilde{D}^{p-1}(h)b_2+h b_1\\
&= \sum_{j=0}^{e} \tilde{D}^{p^j-1}(h)b_{j+1}. %\ \ (b_{e+1} =1).
\end{align*}

First we shall show the following two lemmas concerning the $B$-derivation of $A$. 

\begin{lem}\label{lemd1}
If $\delta$ is a derivation of $A$, then  
$$
\delta(x^k)= \sum_{j=0}^{k-1}\binom{k}{j} x^j \tilde{D}^{k-1-j}(\delta(x)) \ \ for \ k \geq 2.
$$
\end{lem}

\begin{proof}
 We shall show it by an induction. 
Noting that $\delta(x)x=x\delta(x) + \tilde{D}(\delta(x))$,  it is true when $k=2$. 
Assume that $\delta(x^k)= \sum_{j=0}^{k-1}\binom{k}{j} x^j \tilde{D}^{k-1-j}(\delta(x))$. 
Then we obtain 
\begin{align*}
\delta(x^{k+1}) &= \delta(x^k)x + x^k\delta(x)\\
&= \sum_{j=0}^{k-1}\binom{k}{j} x^j \tilde{D}^{k-1-j}(\delta(x))x +x^k\delta(x)\\
&= \sum_{j=0}^{k-1}\binom{k}{j} x^j \Big\{x\tilde{D}^{k-1-j}(\delta(x)) +\tilde{D}^{k-j}(\delta(x)) \Big\}+x^k\delta(x)\\
&= \sum_{j=0}^{k-1}\binom{k}{j} x^{j+1}\tilde{D}^{k-1-j}(\delta(x)) 
+\sum_{j=0}^{k-1}\binom{k}{j} x^j\tilde{D}^{k-j}(\delta(x))  +x^k\delta(x)\\
&= kx^k\delta(x) + \sum_{j=1}^{k-1}\binom{k}{j-1} x^{j}\tilde{D}^{k-j}(\delta(x))\\
& \ \ \ +  \sum_{j=1}^{k-1}\binom{k}{j} x^j\tilde{D}^{k-j}(\delta(x)) + \tilde{D}^k(\delta(x)) +x^k\delta(x)\\ 
&= (k+1) x^k\delta(x) + \sum_{j=1}^{k-1} \Big\{ \binom{k}{j-1} + \binom{k}{j} \Big\} x^{j}\tilde{D}^{k-j}(\delta(x)) + \tilde{D}^k(\delta(x))\\
&= (k+1) x^k\delta(x) + \sum_{j=1}^{k-1} \binom{k+1}{j} x^{j}\tilde{D}^{k-j}(\delta(x)) + \tilde{D}^k(\delta(x))\\
&= \sum_{j=0}^{k} \binom{k+1}{j} x^{j}\tilde{D}^{k-j}(\delta(x)).
\end{align*} 
This completes the proof. 
\end{proof}

\smallskip

\begin{lem}\label{lemd2}
If $\delta$ is a $B$-derivation of $A$, then $\delta(x) \in V$
and $\tau(\delta(x))=0$. 
Conversely, if $g \in V$ with $\tau(g)=0$, then there exists a $B$-derivation 
$\delta$ of $A$ such that $\delta(x)=g$. 
\end{lem}

\begin{proof}
Let $\delta$ be a $B$-derivation of $A$. 
It can be easily   seen that $\alpha \delta(x) = \delta(x) \alpha$ for any $\alpha \in B$. 
Moreover, by Lemma \ref{lemd1}, 
we have 
\begin{align*}
0= \delta(\sum_{j=0}^{e}x^{p^j}b_{j+1} +b_0)=\sum_{j=0}^{e}\delta(x^{p^j})b_{j+1} = \sum_{j=0}^{e}\tilde{D}^{p^j-1}(\delta(x))b_{j+1}=\tau(\delta(x)).
\end{align*}

The converse can be proved  by  a similar way of the proof of Lemma \ref{lemr1}. 
This completes the proof. 
\end{proof}

\smallskip

Now we shall give a sharpening of \cite[Theorem 4.2.3]{HN} 

\begin{thm}\label{thmd1}
Let $f=X^{p^e}+X^{p^{e-1}}b_e+ \cdots +X^pb_{2} + Xb_1 + b_0$ be in $B[X;D]_{(0)}$. 
Then $f$ is  weakly separable  in $B[X;D]$ if and only if 
$$
\{ g \in V \, | \, \tau(g)=0 \} = \tilde{D}(V).
$$
\end{thm}

\begin{proof} 
Assume that $f$ is weakly separable.  
We see that  $\{ g \in V \, | \, \tau(g)=0 \} \supset \tilde{D}(V)$
because $\tau(\tilde{D}(v))=\tilde{D}(\tau(v))=0$ for any $v \in V$. 
Let $g$ be an element in $V$ such that $\tau(g)=0$. 
By Lemma \ref{lemd2}, we can define a $B$-derivation of $A$ by $\delta(x)=g$. 
Since $f$ is weakly separable, 
$g=\delta(x)=hx-xh=\tilde{D}(h)$ for some $h \in V$. 
Thus $g \in \tilde{D}(V)$.  

Conversely, assume that $\{ g \in V \, | \, \tau(g)=0 \} = \tilde{D}(V)$, 
and let $\delta$ be a $B$-derivation of $A$. 
It follows from Lemma \ref{lemd2} that $\delta(x) \in V$ and $\tau(\delta(x))=0$, and hence  
$\delta(x)=\tilde{D}(h) = hx -xh$ for some $h \in V$. 
Then it is easy to see that $\delta(w)= hw-wh$ for any $w \in A$. 
Therefore $\delta$ is inner. 
This completes the proof. 
\end{proof}

%\bigskip

\smallskip

As a direct consequence of Theorem \ref{thmd1}, we have the follwoing. 

\begin{cor}\cite[Theorem 4.2.3]{HN}
Let $B$ be an integral domain of prime characteristic $p$, and $f=X^p+Xb_1+b_0$ in $B[X;\rho]_{(0)}$.
Then $f$ is weakly separable in $B[X;D]$ if and only if 
$$
\{ c \in B\ \, | \, D^{p-1}(c)+cb_1=0 \} = D(B). 
$$
\end{cor}

\begin{proof} 
Let $h =\sum_{j=0}^{p-1}x^j c_j$ be in $V$. 
Since  $\alpha h = h \alpha$ for any $\alpha \in B$,  we have   
$$
c_i \alpha = \sum_{j=i}^{p-1}\binom{j}{i}D^{j-i}(\alpha)c_j \ \ (0 \leq i \leq p-1). 
$$
This implies $c_{p-2} \alpha = \alpha c_{p-2} + (p-1)D(\alpha) c_{p-1}$, and hence $c_{p-1}=0$. 
Repeating this, we obtain $h=c_0 \in B$, namely, $V=B$. 
Then we have the assertion by Theorem \ref{thmd1}. 
This completes the proof. 
\end{proof}

%\bigskip

\smallskip

In virtue of Theorem \ref{thmd1},  we obtain the following theorem 
concerning the relationship between 
the separability and the weakly separability    
in $B[X;D]$. 

\begin{thm}\label{thmd2}
Let $f=X^{p^e}+X^{p^{e-1}}b_e+ \cdots +X^pb_{2} + Xb_1 + b_0$ be in $B[X;D]_{(0)}$.\par 
$(1)$ $f$ is  weakly separable  in $B[X;D]$ if and only if 
 the following  sequence of $V^{\tilde{D}}$-$V^{\tilde{D}}$-homomorphisms is exact:  
$$
0 \longrightarrow
V^{\tilde{D}}  \overset{{\rm inj}}{\longrightarrow}  V 
\overset{\tilde{D}}{\longrightarrow} V \overset{\tau}{\longrightarrow} V^{\tilde{D}}.
$$

$(2)$ $f$ is separable in $B[X;D]$ if and only if 
 the following  sequence of $V^{\tilde{D}}$-$V^{\tilde{D}}$-homomorphisms is exact:  
$$
0 \longrightarrow
V^{\tilde{D}}  \overset{{\rm inj}}{\longrightarrow}  V 
\overset{\tilde{D}}{\longrightarrow} V \overset{\tau}{\longrightarrow} V^{\tilde{D}}
\longrightarrow 0.
$$
\end{thm}
\medskip

\begin{proof} 
(1) It is obvious by Theorem \ref{thmd1}.\par
(2) If $f$ is separable then $f$ is always weakly separable, 
and therefore it suffices to show that ${\rm Im} \, \tau = V^{\tilde{D}}$. 
By \cite[Theorem 4.1]{I1}, 
$f$ is  separable  in $B[X;D]$ if and only if there exists $h \in V$ such that 
$\tau(h)=1$. 
This implies ${\rm Im} \, \tau = V^{\tilde{D}}$ because $\tau$ is a $V^{\tilde{D}}$-homomorphism.  
This completes the proof. 
\end{proof}

\smallskip

Finally, we shall mention briefly on  weakly quasi-separable polynomials in $B[X;D]$.  
As same as automorphism type,  
N. Hamaguchi and A. Nakajima proved that every polynomial in $B[X;D]_{(0)}$ 
is weakly quasi-parable when $B$ is an integral domain (cf. \cite[Theorem 4.2.1]{HN}). 
More precisely, they showed that it is true for a commutative ring $B$ 
when $D(B)$ contains a non-zero-divisor. 
For an arbitrary ring $B$, we have the following. 
\begin{prop}
$(1)$ If $D(B)$ contains a non-zero-divisor, then every polynomial in $B[X;D]_{(0)}$ is weakly quasi-seprable.\par
$(2)$ Let $f=X^{p^e}+X^{p^{e-1}}b_e+ \cdots +X^pb_{2} + Xb_1 + b_0$ be in $B[X;D]_{(0)}$. 
If $b_1$ is a non-zero-divisor in $B$, then $f$ is wakly quasi-separable. 
\end{prop}

\begin{proof} 
(1) Let $g$ be in $B[X;D]_{(0)}$,  $\delta$ a central $B$-derivation of $B[X;\rho]/gB[X;D]$, 
and $x=X+gB[X;D] \in B[X;\rho]/gB[X;D]$. 
Then, for any $\alpha \in B$, we see that 
$$\delta(x\alpha)x=\delta(x)(x\alpha + D(\alpha))=x\delta(x\alpha) + \delta(x)D(\alpha).$$ %\ {\rm that  \ is}, \ \delta(x)D(\alpha)=0.$$
Hence we obtain $\delta(x)D(\alpha)=0$. 
Thus $\delta=0$ if  $D(B)$ contains a non-zero-divisor. 

(2) Let $\delta$ be a central $B$-derivation of $A$. Then 
$\delta(x) \in V$ and $\tau(\delta(x))=0$ by Lemma \ref{lemd2}. 
Since $x \delta(x)= \delta(x) x$, we see that $\delta(x) \in V^{\tilde{D}}$. 
Then we have
\begin{align*}
0= \tau(\delta(x))=\sum_{j=0}^{e} \tilde{D}^{p^j-1}(\delta(x))b_{j+1} = \delta(x)b_1.
\end{align*}
Hence if $b_1$ is a non-zero-divisor in $B$, then $\delta=0$. 
This completes the proof. 
\end{proof}

\smallskip

\section*{Acknowledgement}
The author wishes to thank my supervisor S. Ikehata 
with whose guidance and encouragement this work was done. 
The author also would like to thank the referee for his valuable comments.

\end{document}